\documentclass[12pt,reqno]{article}

\usepackage[margin=1in]{geometry}
\usepackage{setspace}

\usepackage{graphicx}
\usepackage{amscd}
\usepackage{color}
\usepackage{graphics,amsmath,amssymb}
\usepackage{amsthm}
\usepackage{amsfonts}

\usepackage{float}
\usepackage{tikz} 
\usetikzlibrary{shapes,arrows,chains}
\usetikzlibrary[calc]
\usepackage{forest}
\usepackage{subfigure}

\usepackage{verbatim}
\usepackage{enumerate} 
\usepackage{hyperref}
\usepackage{listings}

\theoremstyle{plain}
\newtheorem{theorem}{Theorem}
\newtheorem{corollary}[theorem]{Corollary}
\newtheorem{lemma}[theorem]{Lemma}
\newtheorem{proposition}[theorem]{Proposition}
\newtheorem{conjecture}[theorem]{Conjecture}

\theoremstyle{definition}

\newtheorem{example}[theorem]{Example}

\theoremstyle{remark}
\newtheorem{remark}[theorem]{Remark}

\newcommand{\FF}{\mathcal{F}}
\newcommand{\ZZ}{\mathbb{Z}}

\begin{document}

\title{Ending States of a Special Variant of the Chip-Firing Algorithm}
\author{Tanya Khovanova and Rich Wang}
\date{}

\maketitle

\begin{abstract}
We investigate a special variant of chip-firing, in which we consider an infinite set of rooms on a number line, some of which are occupied by violinists. In a move, we take two violinists in adjacent rooms, and send one of them to the closest unoccupied room to the left and the other to the closest unoccupied room to the right. We classify the different possible final states from repeatedly performing this operation. We introduce numbers $R(N,\ell,x)$ that count labeled recursive rooted trees with $N$ vertices, $\ell$ leaves, and the smallest rooted path ending in $x$. We describe the properties of these numbers and connect them to permutations. We conjecture that these numbers describe the probabilities ending with different final states when the moves are chosen uniformly.
\end{abstract}

\pagenumbering{arabic} 
\section{Introduction}

The chip-firing process on a number line, in which one takes some number of chips at position $i$ and ``fires'' some of them to position $i-1$, and some of them to position $i+1$, has been studied extensively. It was first described by Spencer \cite{chipFireStart} as a balancing game on a collection of vectors in the max norm. In Spencer's variant of chip-firing, one would take all of the chips in a pile and split them as evenly as possible among the neighboring piles. The classical variant of chip-firing, in which one takes two chips from a pile and ``fires'' one to the pile on its left and one to the pile on its right, was later explored by Anderson, Lov\'asz, Shor, Spencer, Tardos, and Winograd \cite{chipFire2}. 

The chip-firing process was found to be generalizable to many other fields and was moved from a line to any graph. Bak, Tang, and Wiesenfeld \cite{abelianSandpile1} examined chip-firing on a graph in the context of dynamical systems, finding the process to be related to systems like pendulums and the stock market. Dhar \cite{abelianSandpile2} expanded upon this by showing that their model satisfied abelian dynamics and coined the term ``abelian sandpile model.''

One important property of chip-firing algorithms is \textit{confluence}, or the property where the final state does not depend on the order of moves performed, which was first analyzed by Bj\"orner, Lov\'asz, and Shor \cite{chipFire3}, and is a consequence of Newman's Lemma \cite{confluence1}. This property can be found in both the classical chip-firing algorithm and in many modifications, such as the one analyzed by Hopkins, McConville, and Propp \cite{sortChipFire}. 

However, a slight modification to the classic chip-firing process changes the situation significantly. We instead take two chips at positions $i$ and $i+1$ and ``fire'' them to the closest positions to their left and right, respectively, that do not have a chip. This new chip-firing variant is not confluent, which makes it interesting to study. Standard chip-firing operations are usually local, meaning that each move they perform changes the position of the chips by a fixed distance. Our dynamical system can move a chip arbitrarily far away from the original position. Darij Grinberg suggests to call this operation \textit{a two-sided dispersion}.

To describe our process, we use the analogy of violinists staying in hotel rooms. The following problem is in the center of our paper.

\begin{quote} 
\textbf{Problem.} Consider a hotel with an infinite number of rooms arranged sequentially on the ground floor. The rooms are labeled from left to right by integers, with room $i$ being adjacent to rooms $i - 1$ and $i + 1$. Room $i-1$ is on the left of room $i$, while room $i+1$ is on the right.

A finite number of violinists are staying in the hotel; each room has at most one violinist in it. Each night, some two violinists staying in adjacent rooms (if two such violinists exist) decide they cannot stand each other's noise and move apart: One of them moves to the nearest unoccupied room to the left, while the other moves to the nearest unoccupied room to the right. This keeps happening for as long as there are two violinists in adjacent rooms.
Prove that this moving will stop after a finite number of days.
\end{quote}

\begin{example}\label{ex:fClusteron4}
Suppose we have four violinists with their initial placement 0001111000. From the initial state, we have three options, which we call L, M, and R. Regardless of which is chosen, in the next state, we have two options, which we call L and R. After that, we have only one move in each case. This is shown in Figure~\ref{fig:4exampletree}.

\begin{figure}[ht!]
    \centering                      
        \scalebox{0.7}{
            \begin{forest}
                [$0001111000$
                [$0010011100$
                [$0010100110$[$0010101001$]]
                [$0010110010$[$0011001010$[$0100101010$]]]
                ]
                [$0011001100$
                [$0100101100$[$0100110010$[$0101001010$]]]
                [$0011010010$[$0100110010$[$0101001010$]]]
                ]
                [$0011100100$
                [$0100110100$[$0101001100$[$0101010010$]]]
                [$0110010100$[$1001010100$]]
                ]
                ]
            \end{forest}
        }
        \caption{Tree diagram for possible moves for 4 violinists.}
    \label{fig:4exampletree}
\end{figure}
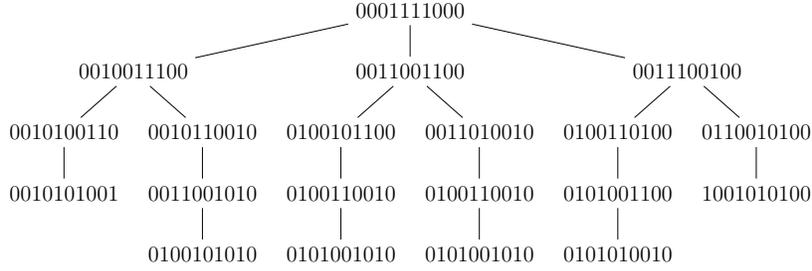
\end{example}

In Figure~\ref{fig:4exampletree}, we can also see that there are $5$ different final states. The lack of a fixed final state opens up many possible areas of study on the possible final states.

We start with the solution to the problem in Section~\ref{sec:solution}. 
We introduce some important definitions, including the centroid and the sumtroid, and cover the relevant preliminary results in Section~\ref{sec:preliminaries}. We also introduce an object of our study--- a clusteron, which is an initial state that consists of a set of consecutive occupied rooms and allows multiple occupancy.

In Section~\ref{sec:bijection}, we show that starting from a clusteron, we always achieve a final state with single occupancy rooms. We also introduce a bijection from states with only one violinist per room to a state that considers the sequence of clusterons in that state.

We characterize all possible final states of a single clusteron in Section~\ref{sec:fsclusterons}. In particular, we show that all final states have a single gap of size $2$ and the rest of their gaps size $1$. We also show that final states are uniquely determined by their centroids.

Section~\ref{sec:interstates} is devoted to intermediate states.

In Section~\ref{sec:fsprobability}, we conjecture that if possible moves are made to be equally probable for any clusteron with single-occupancy rooms, then the probability of achieving each of the claimed final states from Section~\ref{sec:fsclusterons}, up to translation, is the same. In particular, in Figure~\ref{fig:4exampletree}, one can see that if one starts at the top and traverses down each branch of the tree with equal probability, then the probabilities of obtaining a state congruent to $10101001$, $10100101$, and $10010101$ up to translation and ignoring zeroes on the boundary, are equal, $\frac{1}{3}$ in each case.

In Subection~\ref{sec:recursivetrees}, we introduce numbers $R(N,\ell,x)$ that count the number of labeled recursive routed trees (also known as increasing Cayley trees as seen in Stanley \cite{Stanley}) with $N$ vertices, $\ell$ leaves, and with the smallest rooted path ending in $x$. We conjecture that $R(N,\ell,x)$ equals $(N-1)!$ times the probability that the initial state with $N$ violinists ends in a final state with a particular sumtroid which can be defined in terms of $\ell$ and $x$. We prove a recursive formula for numbers $R(N,\ell,x)$.

In Section~\ref{sec:permutations}, we show that the numbers $R(N,\ell,x)$ are the same numbers as studied by Conger \cite{Conger} and that they are equal to the number of permutations of $N-1$ with $\ell-1$ descents and last digit $x-1$.

Section~\ref{sec:additionaldata} contains additional data.

\section{Solution to the Problem}\label{sec:solution}

We denote the total number of violinists as $N$. Let us regard the violinists as indistinguishable (i.e., we do not care which violinist is in which room, but only care about which rooms are occupied). Thus, we can identify any state with a finite subset of integers. 

Since we only care about which rooms are occupied, we can redefine the operation as follows: Choose two adjacent violinists $v_1,v_2$, and say $v_1$ is to the left of $v_2$. We move $v_1$ into the room directly on his left. If this room is currently occupied by some violinist $v_3$, then we move $v_3$ into the room on his left as well. And so on, until a violinist that we move enters an unoccupied room. A similar operation will occur for $v_2$, but moving to the right. This way the order of the violinists from left to right is always preserved. If there are multiple violinists in a room, then we consider them to be ordered left to right in some manner. We call this adjusted procedure \textit{chip-pushing}.

When considering the operation as chip-pushing, the leftmost violinist can only move to the left, while the rightmost only to the right. Moreover, the leftmost violinist will continue to be the leftmost violinist in all future states. A similar statement holds true for the rightmost violinist.

We define a \textit{gap} to be the number of unoccupied rooms between two occupied rooms that do not have any other occupied rooms between them, and an \textit{x-gap} to be a gap of size $x$.
  
We define the \textit{span} of a state to be the number of rooms between the leftmost and rightmost violinists, including the rooms that these two violinists occupy.

\begin{theorem}
The process of violinists moving to different rooms cannot continue indefinitely.
\end{theorem}

The proof is known \cite{DarijGrinberg2021}, but we want to present our solution utilizing the idea of chip-pushing.

\begin{proof} 
We prove that there exists an upper bound on the maximum possible number of times that the chip-pushing operation can be performed on a state with $N$ violinists by using induction. If we have $1$ violinist, then no operations can be performed. If we have $2$ violinists, then the final state is achieved after one move. Now we show that if there exists an upper bound $b_N$ for $N$ violinists, there also exists an upper bound for the number of moves made by $N+1$ violinists.

We think of our operation as chip-pushing. Thus, the order of the violinists is preserved. We denote the $i$-th violinist from the left to be violinist $v_i$. Consider the leftmost and rightmost violinists on the number line, $v_1$ and $v_{N+1}$, respectively. If our process does not terminate for $N+1$ violinists, then one of $v_1$ or $v_{N+1}$ must make an infinite number of moves, or else the process will end by our inductive assumption. So say without loss of generality that $v_{N+1}$ makes an infinite number of moves. After $v_{N+1}$ moves $2Nb_N+1$ times, the span of any state following these moves must be greater than $2N b_N+1$. Thus, by the pigeonhole principle, there exists a gap of size at least $\left\lceil\frac{2N b_N+1}{N}\right\rceil=2b_N+1$, between some two violinists $v_i$ and $v_{i+1}$, for some $1\le i \le N$. If we consider the sets of violinists $S_1=\{v_1,v_2,\dots,v_i\}$ and $S_2=\{v_{i+1},v_{i+2},\dots,v_{N+1}\}$, then within each set, at most $b_i$ and $b_{N+1-i}$ moves can be performed, respectively, both of which are less than or equal to $b_N$. As a result, because the gap between $v_i$ and $v_{i+1}$ has size at least $2 b_N+1>b_{N+1-i}+b_{i}$, they will never be adjacent. This means that no more moves can be performed with one violinist from $S_1$ and one violinist from $S_2$. But the number of moves that can be performed within each set is bounded by constants $b_i$ and $b_{N+1-i}$. Thus, the total number of operations for $N+1$ violinists is bounded as well.
\end{proof}

\section{Preliminaries} \label{sec:preliminaries}

In this section we define a few terms that will be used throughout the paper. We allow several violinists in a single room.

We are interested in studying the initial states where violinists occupy some number of consecutive rooms. We also allow having several violinists per room. We call such initial states \textit{clusterons}. The original problem corresponds to a specific case when each room is initially occupied by a single violinist. We call such states \textit{flat clusterons}. Example~\ref{ex:fClusteron4} shows all possible moves starting from a flat clusteron with $N=4$ violinists.

We define the \textit{size} of a clusteron as the total number of violinists occupying some room in it. The size of a clusteron will be denoted by $N$, where $N>0$.

Suppose $a_i$ is the number of violinists in room $i$ in state $S$. Then we build the Laurent polynomial
\[P(S,t) = \sum a_i t^i.\]
We can express the total number of violinists $N$ as $P(S,1)$.

The \textit{entropy} of state $S$ is defined as $P(S,2)$; see \cite{DarijGrinberg2021}.

We define the \textit{centroid}, $C(S)$ as
\[C(S) = \frac{1}{N} \sum a_i i.\]

Let $K=NC$ denote the \textit{sumtroid} of a state, where $C$ is the centroid of the state. In particular, $K$ also equals the sum of the room numbers of all occupied rooms. 

Equivalently, 
\[C(S) = \frac{1}{N} P'(S,1) \quad \textrm{ and } \quad K = P'(S,1).\]

Suppose we can make a move on the rooms $(i,i+1)$, such that $\ell-1$ rooms to the left of room $i$ are occupied, and $r-1$ rooms to the right of room $i+1$ are occupied. We call $\ell$ and $r$ the \textit{left} and \textit{right neighborhoods} of the move, respectively.

The following proposition in case of flat clusterons was proven in \cite{DarijGrinberg2021}. We repeat the proof for a more general case of any state that allows a move.

\begin{proposition}\label{prop:increaseEntropy}
Given any state that has a move, the entropy of a state increases after a move.
\end{proposition}

\begin{proof}
Given a move with $\ell$ and $r$ being left and right neighborhoods, the state polynomial changes by $\left(t^{i-\ell}+t^{i+1+r}\right)-\left(t^i+t^{i+1}\right)$.
The entropy of the state changes by $\left(2^{i-\ell}+2^{i+1+r}\right)-\left(2^i+2^{i+1}\right)$. Because $2^{i+1+r}>2^i+2^{i+1}$ for all values of $i$, this change is always positive.
\end{proof}

In particular, given a move with $\ell$ and $r$ being left and right neighborhoods, the centroid changes by $\frac{r-\ell}{N}$, and the sumtroid changes by $r-\ell$. This implies the following lemma.

\begin{lemma}\label{lem:2clusteronconstantcentroid}
The centroid (sumtroid) remains constant after a particular move if and only if the right and left neighborhoods of the move are the same.
\end{lemma}

\section{Initial Results and Bijection to Consecutive Groups of Rooms}\label{sec:bijection}

Our definition of the centroid is dependent on labeling the rooms. If we change the label of every room from $a$ to $a+x$, where $x$, the centroid of the state will increase by $x$ too. We allow $x$ to be an integer or half-integer. If $x$ is a half-integer, the new labeling for rooms becomes by half-integers too. By default, we fix the labeling so that the centroid (and sumtroid) of any starting state is $0$. If $N$ is odd, the labels are integers. If $N$ is even, the labels have a fractional part equal to $\frac{1}{2}$.

We continue with Example~\ref{ex:fClusteron4} once again.
\begin{example}
Consider the starting state with 4 violinists in rooms $-\frac{3}{2}$, $-\frac{1}{2}$, $\frac{1}{2}$, and $\frac{3}{2}$. For the first move, we have three choices L, M, and R (for left, middle, and right). For the second move, we have two choices L and R. Table~\ref{table:fClusteron4} shows the achievable final states, the moves that lead to them, and their corresponding sumtroids. Notice that one of the final states can be achieved by two different sequences of moves.

\begin{table}[ht]\label{table:fClusteron4}
    \centering
    \begin{tabular}{c|c|c}
        Final State & Moves & Sumtroid \\
        \hline
        0010101001 & LL & $3$ \\
        0100101010 & LR & $1$ \\
        0101001010 & ML, MR & $0$\\
        0101010010 & RL & $-1$\\
        1001010100 & RR & $-3$
    \end{tabular}
    \caption{Final states from flat clusteron of size $4$ with sumtroid values.}
    \label{tab:4fs}
\end{table}
\end{example}

In our new setup, when we allow several violinists in a room, it is not clear that we always have an available move. If we start with one room with several violinists, the moves are not defined. We call a room an \textit{isolated room} if both of its neighboring rooms are empty. We also call a room a \textit{crowded room} if it has more than one violinist in it. 

Theoretically, it might be possible to start with a clusteron and, after some number of moves, reach an isolated room with more than one violinist. We do not have moves for this case. However, Proposition~\ref{prop:multipleoccupancy} shows that if we start with a clusteron with more than one room, we do not get into such a predicament.

\begin{proposition} \label{prop:multipleoccupancy}
Consider a state that originates from a clusteron with size greater than $1$. From this state, it is impossible to reach a state in which there exists a crowded isolated room.
\end{proposition}
\begin{proof}
Assume for the sake of contradiction that there exists some sequence of moves for a clusteron that creates an isolated room $r_i$, and let $a_1$, $a_2$, $\dots$, $a_j$ be the shortest such sequence. Move $a_j$ must include one of rooms $r_i-1$ and $r_i+1$, or $a_1$, $a_2$, $\dots$, $a_{j-1}$ would be a shorter sequence.
    
Thus, $a_j$ must move violinists in one of the following pairs of rooms: $(r_i-2,r_i-1)$, $(r_i-1,r_i)$, $(r_i,r_i+1)$, or $(r_i+1,r_i+2)$. But moving violinists in rooms $(r_i-2,r_i-1)$ moves the violinist in room $r_i-1$ to the closest unoccupied room to his right. However, room $r_i$ is already occupied, meaning that the violinist from room $r_i-1$ will either move into room $r_i+1$, or move past room $r_i+1$. In either case, room $r_i+1$ will be occupied after $a_j$, creating a contradiction. Similarly, the move $(r_i-1,r_i)$ will move one of the violinists in room $r_i$ to the closest unoccupied room to the right of room $r_i$, guaranteeing that $r_i+1$ is occupied after $a_j$. After this move, both rooms $r_i$ and $r_i+1$ would still be occupied, and we would be able to perform another move, contradiction. By symmetry, moving one of the pairs $(r_i,r_i+1)$ or $(r_i+1,r_i+2)$ would result in room $r_i-1$ being occupied after $a_j$, contradiction.
\end{proof}

\begin{theorem}
All final states of a clusteron of size greater than $1$ have single occupancy rooms.
\end{theorem}

\begin{proof}
By Proposition~\ref{prop:multipleoccupancy}, if we are ever in a state in which there exists a room with multiple occupants, then one of its neighboring rooms must be occupied, and thus we cannot be in a final state. 
\end{proof}

We define a \textit{flat near-clusteron} to be a finite consecutive set of single occupancy rooms with exactly one unoccupied room between the leftmost and rightmost occupied rooms in the set. The size of a flat near-clusteron is defined as the number of occupied rooms in it.

We define a \textit{propagated state} of a clusteron to be any state that has originated from a clusteron.

Before concluding this section, we give an alternate representation of the problem in terms of flat clusterons, which allows us to describe states in shorter form.

Define a \textit{suite} to be the inclusion-maximal set of occupied rooms. To each suite, we assign a number equal to the number of violinists in the suite. In a suite-state, a run of $i$ consecutive zeroes corresponds to $i+1$ empty rooms in a room-state. Any nonzero number $x$ in a suite-state corresponds to a run of $x$ consecutively occupied rooms in a room-state.

Consider a bijection $B$ between single occupancy room-states and suite-states, where we replace a run of ones with the number of ones in a run, and remove one of the zeros between two consecutive runs of ones. For example, $B(1011001) = 1201$.

We define a move on a state of suites as follows. Whenever we have a suite with more than $k>1$ violinists, send $0<x<k$ occupants to the suite directly to the left and $k-x$ occupants to the suite to the right, leaving $0$ violinists in the original suite.

We define the formula for the centroid of a state of suites to be the same as the formula to obtain the value of the centroid of a state of rooms.

We now show that, given a room-state $S_1$ and a move from $S_1$ to $S_2$, there exists a move on $B(S_1)$ such that the resulting state is $B(S_2)$. We also show that if we have states of rooms $S_1$ and $S_2$ and states of suites $S'_1$ and $S'_2$ such that $B(S_1)=S'_1$ and $B(S_2)=S'_2$, and there exists a move that can be performed to $S_1$ to turn it into state $S_2$, then the corresponding move turning $S'_1$ into $S'_2$ will have the same impact on the value of centroid of $S'_1$ as the original move does on the value of the centroid of $S_1$.

\begin{theorem} \label{thm:suites}
The bijection $B$ preserves moves and centroids.
\end{theorem}

\begin{proof}
First, we show that $B$ preserves moves. Consider a room-state $S$. Suppose our move is from the $k$-th run of consecutive rooms, which has $x+y$ consecutive rooms, and splits the run into two separate runs of $x$ and $y$ consecutive rooms, creating a state $S'$. Performing the corresponding move on $B(S)$, which is splitting the $k$-th nonempty suite of $B(S)$ by moving $x$ of the violinists there into the room on the left and $y$ of them into the room on the right, creates the suite-state $B(S')$, as when moving from $S$ to $S'$, the $x$ violinists who move to the left will join the run of consecutive rooms to their left if they were previously separated by a single empty room, and be their own run of consecutive rooms if not. Similar arguments can be made for the $y$ violinists who move to the right.

The operation turning a move on a room-state into a move on a suite-state is reversible, which proves the reverse direction.

Now we show that $B$ preserves centroids. All $B$ does is switch the number of people in each room and the distance each violinist travels. Because the change in the centroid for each move is calculated by the quantity
$$\frac{\text{distance moved}\cdot\text{number of violinists that move that distance}}{\text{total number of violinists}},$$
and $B$ preserves the value of both the numerator and denominator, we conclude.
\end{proof}

For an example of how moves are preserved, the moves for the new action with suites in Figure ~\ref{fig:fClusteron4New} correspond to the old action in a flat clusteron of size $4$ in Figure~\ref{fig:4exampletree}.

\begin{figure}[ht!]
    \centering
        \scalebox{0.6}{
            \begin{forest}
             [$4$[$103$[$1102$[$11101$]][$1201$[$2011$[$10111$]]]][$202$[$1012$[$10201$[$11011$]]][$2101$[$10201$[$11011$]]]][$301$[$1021$[$1102$[$11101$]]][$2011$[$10111$]]]]
            \end{forest}
        }
        \caption{New action, with suites.}
    \label{fig:fClusteron4New}
\end{figure}
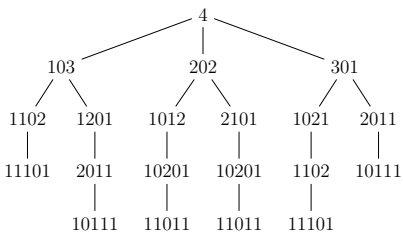

\section{Final States of Clusterons}\label{sec:fsclusterons}

For any state $S$, let us denote the set of all occupied rooms together with all neighboring rooms as $S'$. We have that $|S'| \leq 3 |S|$. We also define $S_i$ to be the state after the $i$-th move has been made.

\begin{proposition}[Grinberg \cite{DarijGrinberg2021}]\label{prop:occupiedRooms}
If a room is occupied or neighbors an occupied room, then the same is true for future states: $S'_i \subset S'_{i+1}$. 
\end{proposition}

Using this, we can prove a statement on the sizes of gaps in final states.

\begin{proposition}\label{prop:gapsz12}
All final states of a clusteron only have gaps of size 1 or 2.
\end{proposition}

\begin{proof}
No final state can have a gap of size $0$, as then it would still be possible to perform a move. Assume for the sake of contradiction it is possible to have a gap of size larger than $2$, and that some sequence of moves results in a gap of size larger than $2$ between rooms $r_i$ and $r_{i}+j$, for some $j>3$. Notice that rooms $r_{i}+1$, $r_{i}+2$, and $r_{i}+3$ are empty.

Because a flat clusteron has an initial state with one violinist in multiple consecutive rooms, rooms $r_i$ through $r_{i}+j$ must all have been occupied by a violinist at some point due to the way that our chip-pushing operation is defined. But by Proposition~\ref{prop:occupiedRooms}, this means that room $r_{i}+2$ should either be occupied by or neighbor a room occupied by a violinist in the final state, contradiction.
\end{proof}

Our goal is to describe all possible final states of clusterons. We start by studying final states up to translation. In other words, given a state $S$, we only consider the segment of rooms starting from the first occupied room to the last occupied room and disregard the absolute index of each room, only caring about the relative distances. We call this the $\textit{shadow}$ of $S$.

We define the shadow $F_{N,k}$ to be the shadow that has $N$ violinists, $N-2$ gaps of size $1$, and $1$ gap of size $2$, which is in between the $k$-th and $(k+1)$-st occupied rooms. For example, the shadow $F_{3,2}$ is $101001$. 

We introduce the set $\mathcal{F}_N$, which depends on $N$, to be the set of all shadows with $N$ violinists that have $N-2$ gaps of size $1$ and $1$ gap of size $2$. Equivalently, we can write
$$\{F_{N,i} \mid 1\le i \le N-1\}=\mathcal{F}_N.$$

Our goal is to prove the following theorem.

\begin{theorem}\label{thm:fsclusterons}
The set $\mathcal{F}_N$ equals the set of all final shadows of a clusteron for all $N>1$, except for the clusterons $12$ and $21$.
\end{theorem}

For the clusteron $12$, there only exists one possible set of moves before reaching an ending state: $12 \rightarrow 1011 \rightarrow 11001 \rightarrow 100101$. A similar statement holds true for the clusteron $21$. Both of these can only reach one of the final shadows in the set $\mathcal{F}_N=\{100101,101001\}$.

The theorem is equivalent to proving that all states with span $2N$ and no two violinists in adjacent rooms are possible final states and the only possible final states for all clusterons except for the two special cases in the example above.

Before we can prove Theorem~\ref{thm:fsclusterons}, we first need to prove a few lemmas.

By Proposition~\ref{prop:gapsz12}, we know that gaps in the final state can only be of size $1$ or $2$. We plan to prove that all final states have exactly one gap of size 2. However, an intermediate state might have several gaps of size 2.

\begin{lemma} \label{lem:clusteron2gaps}
In any propagated state $S$ from a clusteron, there will always exist at least one pair of two consecutive occupied rooms between any two gaps of size $2$ in $S$.
\end{lemma}

\begin{proof}
Assume the contrary. Let state $S$ be the first state where there does not exist a pair of two consecutive occupied rooms between two gaps of size $2$. Say that in state $S$, the closest two gaps of size $2$ satisfying this property are at rooms $(i,i+1)$ and $(j,j+1)$, where $i+1<j-1$. Because this is the first such state, the move made right before reaching state $S$ must have been performed on two violinists residing between rooms $i$ and $j+1$, inclusive. 

Say that the move was made at positions $p,p+1$. This creates a 2-gap at rooms $(p,p+1)$. If $i<p<j$, then this contradicts the fact that $(i,i+1)$ and $(j,j+1)$ are the closest two gaps of size $2$, as then $(p,p+1)$ would be closer to $(i,i+1)$.

But if $p=i$, then consider the state before $S$, namely $S_0$. Because in $S$ there are no consecutively occupied rooms between $(i,i+1)$ and $(j,j+1)$, this means that in $S_0$, rooms $i+2$ and $i+3$ are unoccupied. Otherwise, after performing a move on rooms $(i,i+1)$, the violinist from room $i+1$ would have to go to a room next to another occupied room. However, this would imply that in $S_0$, the 2-gaps $(p,p+1)$ and $(j,j+1)$ would not have any consecutively occupied rooms between them, contradicting the fact that $S$ is the first state with this property. Similar reasoning shows that $p=j$ fails as well.

As such, all cases are exhausted, and we have a contradiction. All gaps of size $2$ must always have at least one pair of two consecutive occupied rooms between them.
\end{proof}

This is enough to show that any final state will have at most one gap of size $2$, or else by Lemma~\ref{lem:clusteron2gaps} it would not be a final state. All final states must also have at least one gap of size $2$, as making a move automatically creates a gap of size $2$ in the subsequent state. As a result, we have the following corollary.

\begin{corollary} \label{col:clusteron2gaps}
All final states from a clusteron of size $N>1$ must have exactly one gap of size $2$ and all single occupancy rooms.  
\end{corollary}

We have now shown that all shadows of the final states belong to $\mathcal{F}_N$. Our next goal is to show which states are achievable.

\begin{lemma} \label{lem:clusteronBridge}
It is always possible to reach a flat clusteron or flat near-clusteron when starting from a clusteron of size $N>1$.
\end{lemma}

\begin{proof}
By Proposition~\ref{prop:multipleoccupancy}, it is possible to continually perform moves on pairs of rooms where at least one room acted upon in each move has more than one occupant until we reach a state in which every room has less than or equal to $1$ violinist.

We claim that while performing this set of moves, there will exist at most one gap in all propagated states from the clusteron, which, if it exists, will have size $1$.

Assume for the sake of contradiction that state $S$ is the first propagated state to violate this claim, and it comes after state $S_0$.

If $S_0$ did not have any gaps, then we have two cases. If the move we perform to get from $S_0$ to $S$ is performed on two crowded rooms, then $S$ will not have any gaps either. If the move is performed on one crowded room and one single-occupancy room, then $S$ will have exactly $1$ gap, which will be of size $1$. Both cases give a contradiction.

On the other hand, if $S_0$ had exactly $1$ gap, which was of size $1$, then we would once more have two cases. If the move we perform to get from $S_0$ to $S$ is performed on two crowded rooms, then $S$ will not have any gaps, as the gap previously present in $S_0$ will get filled by one of the violinists acted upon by our move. If the move is performed on one crowded room and one single-occupancy room, then $S$ will have exactly $1$ gap, which will be of size $1$, as the gap previously present in $S_0$ will again get filled by one of the violinists acted upon by our move, and the single-occupancy room our move acted upon will create a new gap of size $1$. Once again, both cases give a contradiction.
\end{proof}

\begin{lemma} \label{lem:flatStateAchievability}
For all flat or flat near-clusterons $a_1a_2\dots a_{m}$, it is possible to achieve all possible final shadows in the set $\mathcal{F}_N$, with the exception of $11011$, $1011$, $1101$, and $101$.
\end{lemma}

\begin{proof}
For the flat clusterons or near-clusterons of size $N<5$, we can write out all possible sets of moves, namely for $1$, $11$, $111$, $1111$, $11101$, and $10111$.

The states $1$ and $11$ can easily be seen to achieve all final shadows. For the state $111$, we analyze its suite state, $3$, below.
\[
3 \rightarrow 102 \rightarrow 1101
\]
\[
3 \rightarrow 201 \rightarrow 1011
\]
Similarly, for the state $11101$, we analyze its suite state, $31$, below.
\[
31 \rightarrow 202 \rightarrow 1012 \rightarrow 10201 \rightarrow 11011
\]
\[
31 \rightarrow 103 \rightarrow 1102 \rightarrow 11101
\]
\[
31 \rightarrow 103 \rightarrow 1201 \rightarrow 2011 \rightarrow 10111
\]
The state $10111$ also achieves all final shadows by symmetry. Finally, $1111$ can be seen to achieve all final states in Figure~\ref{fig:4exampletree}.

For $N\ge 5$, we prove the lemma by induction. It can be verified that all final shadows are achievable for all flat clusterons and flat near-clusterons with size $N=5$.

For the base case of $N=5$, we first show that all final shadows are achievable from the room-state $1001111$, which corresponds to the suite-state of $104$ below.
\[
104 \rightarrow 1103 \rightarrow 11201 \rightarrow 12011 \rightarrow 20111 \rightarrow 101111 
\]
\[
104 \rightarrow 1202 \rightarrow 2012 \rightarrow 20201 \rightarrow 101201 \rightarrow 102011 \rightarrow 110111
\]
\[
104 \rightarrow 1301 \rightarrow 2021 \rightarrow 10121 \rightarrow 10202 \rightarrow 11012 \rightarrow 110201 \rightarrow 111011
\]
\[
104 \rightarrow 1103 \rightarrow 11102 \rightarrow 111101
\]

By symmetry, all final states are also achievable from the room-state $1111001$.

When starting with a flat clusteron or a flat near-clusteron of size $5$, we can always find two consecutive occupied rooms at the beginning or end of that clusteron. By making a move on these two rooms, we can reach either the room-state $1001111$ or $1111001$. We can conclude that all flat and flat near-clusterons of size $5$ can also achieve any final state.

For the inductive step, we consider the operation as chip-pushing. We assume the result for $N$ and prove it for $N+1$. 

For any flat clusteron or near-clusteron of size $N+1$, consider the set of violinists excluding the one on the very left. This set of violinists is a flat clusteron or near-clusteron of size $N$. By our inductive assumption, we can make moves to turn this clusteron into any shadow in $\mathcal{F}_N$. But because we are considering our operation as chip-pushing, after performing the moves to create that shadow in $\mathcal{F}_N$, the leftmost violinist will get pushed such that it is in the room directly to the left of the second-leftmost violinist. In particular, this means that we can create all shadows of size $N+1$ that consist of some violinist directly to the left of any shadow of size $N$, and by symmetry, any shadow that consists of some violinist directly to the right of any shadow of size $N$.

Placing a violinist directly to the left of shadow $F_{N,i}$ and performing the only possible sequence of moves until a final state is reached will create the shadow $F_{N+1,i}$, and placing a violinist directly to the right and performing the only possible sequence of moves will create the shadow $F_{N+1,i+1}$. Since $i$ ranges from $0$ to $N-1$, we can see that all $F_{N+1,i'}$ can be created via one of these transformations for $0\le i' \le N$, so we are done.
\end{proof}

\begin{remark}
The exceptions listed in Lemma~\ref{lem:flatStateAchievability} fail because there are certain final shadows they cannot reach. The state $101$ cannot reach the final shadow $1001$, the state $1011$ cannot reach $101001$, the state $1101$ cannot reach $100101$, and $11011$ cannot reach $10100101$.
\end{remark}

We can now completely analyze all clusterons of size 4 and show that we can achieve all final shadows for each.

\begin{example}\label{ex:allShadowsN=4}
Suppose $N=4$. Up to symmetries, all possible clusterons are 1111, 121, 211, 22, 21. The clusteron 31 after one move becomes 121, and the clusteron 22 after one move becomes 1111. The 1111 case is completely analyzed at the beginning of the paper in Example~\ref{ex:fClusteron4}.

We are left to analyze cases 121 and 211.

Case 211 has the following possibilities:
$$211 \rightarrow 11011 \rightarrow 100111 \rightarrow 1010011 \rightarrow 10101001$$
$$211 \rightarrow 12001 \rightarrow 101101 \rightarrow 110011 \rightarrow 1001011 \rightarrow 10011001 \rightarrow 10100101$$
$$211 \rightarrow 11011 \rightarrow 111001 \rightarrow 1100101 \rightarrow 10010101.$$

A move can be made on the clusteron $121$ to turn it into the room-state $11101$, which we have shown in Lemma~\ref{lem:flatStateAchievability} can achieve all final shadows.

Case 10111 can lead to the following possibilities:
$$10111 \rightarrow 111001 \rightarrow 1100101 \rightarrow 10010101$$
$$10111 \rightarrow 110011 \rightarrow 1001011 \rightarrow 10011001 \rightarrow 10100101.$$
By symmetry, the final shadow 10101001 can be achieved from case 11101.
\end{example}

We are now ready to prove Theorem~\ref{thm:fsclusterons}.
 
\begin{proof}[Proof of Theorem~\ref{thm:fsclusterons}]

By Corollary~\ref{col:clusteron2gaps}, we showed that all possible final shadows are in $\mathcal{F}_N$. 

For $N\ge 5$, the fact that every shadow is achievable is immediate after combining Lemma~\ref{lem:clusteronBridge} and Lemma~\ref{lem:flatStateAchievability}.

We are left with cases of small $N$. We can directly check that all shadows are achievable for $N=1,2,3,4$.
\end{proof}

Being able to describe all final shadows in such a manner has some immediate consequences on the centroids of final states.

\begin{proposition}\label{prop:centroidEnumeration}
All distinct final states have unique centroids.
\end{proposition}
\begin{proof}
If the leftmost violinist is in room $r$, and the double is the $k$-th gap, then the centroid can be computed to be:
\[r+N - \frac{k}{N}.\]
The centroid is distinct for each pair $(r,k)$.
\end{proof}
\begin{corollary}\label{col:cogcongruency}
Final states are congruent up to translation if and only if the centroid has the same fractional part.
\end{corollary}

We have been able to describe all final shadows of a clusteron. We now work to extend this result to prove a statement on all final states of a clusteron, starting with Proposition~\ref{prop:N-1leftorright}.

\begin{proposition} \label{prop:N-1leftorright}
In any propagated state from a flat clusteron, the furthest that a violinist can move from its original room is bounded above by $N-1$.
\end{proposition}
\begin{proof}
We use induction on the index of our violinists to show that no violinist can travel more than $N-1$ rooms to its left. Consider our operation as chip-pushing, which imposes an ordering on violinists that is preserved across moves, and say that the first violinist is in room $0$. After the first move, some violinist will move into room $N$. By the idea of chip-pushing, the violinist now in room $N$ can never move to a room to the left of room $N$. By Theorem~\ref{thm:fsclusterons}, the span of the final state of any clusteron is exactly $2N$, so the leftmost room that the first violinist can eventually occupy is room $-N+1$, as the first violinist can only travel to the left. This is exactly $N-1$ rooms to the left of room $0$, so we have proven the desired result for the $1$-st violinist. 

This is our base case. For the inductive step, we assume the result for the $i$-th violinist and prove it for the $(i+1)$-st violinist. Because the $(i+1)$-st violinist starts $1$ room to the right of the $i$-th violinist, the $(i+1)$-st violinist must also end at least one room to the right of the $i$-th violinist, and by the inductive assumption the number of times the $i$-th violinist can move to the left is bounded above by $N-1$, the same bound holds for the number of times the $(i+1)$-st violinist can move to the right.

We can similarly prove the desired result for the rightmost room that a violinist can travel to.

To note that $N-1$ is achievable for the leftmost violinist in particular, we can see that performing moves on the sequence of rooms $(N-2,N-1),(N-4,N-2),\dots,(-N+2,-N+3)$ causes the violinist starting in room $0$ to end in room $-N+1$.
\end{proof}

We start by proving a supporting lemma that describes what happens when we place a violinist next to a shadow of a final state.

\begin{lemma}\label{lem:nextToShadow}
Say that we have a state with $N+1$ violinists that has a violinist in some room $k$, and next to it we have a shadow $F_{N,r}$ with the leftmost violinist in room $k+1$ for some $1\le r \le N-1$. Then the only possible final state from here is $F_{N+1,r}$ with the leftmost violinist in room $k-1$.
\end{lemma}
\begin{proof}
We perform moves on the following sequence of rooms: $(k,k+1)$, $(k+2,k+3)$, $\dots$, $(k+2r-3,k+2r-2)$, at which point we can no longer make any moves, and we end up with a final state with shadow $F_{N+1,r}$ with the leftmost violinist in room $k-1$. At each step in the process, only one move can be made, which is why this is the only possible achievable final state.
\end{proof}

Now we are ready to prove the main theorem.

\begin{theorem}\label{thm:clusteronfstrans}
The final states of a flat clusteron of size $N$ starting with the leftmost violinist in room $0$ can be expressed as one of the following:
\begin{itemize}
    \item the shadow $F_{N,1}$ with first violinist in room $-N+1$
    \item the shadow $F_{N,N-1}$ with first violinist in room $-1$
    \item the shadow $F_{N,r}$ for any $1 \le r \le N-1$ with first violinist in room $k$ for all $-N+2\le k \le -2$.
\end{itemize}
\end{theorem}
\begin{proof}
The proof is similar to that of Theorem~\ref{thm:fsclusterons}.

We first show that all claimed final states are achievable.

The edge cases, or obtaining a state with shadow $F_{N,1}$ with the leftmost violinist in room $-N+1$, and obtaining a state with shadow $F_{N,N-1}$ with the leftmost violinist in room $-1$, can be achieved by continually performing the leftmost possible move and rightmost possible move, respectively.

To show that all other claimed states are obtainable, we proceed by induction. The base cases of $N=1,2,3,4$ can be manually verified. For the inductive step, say that we have proven the result for $N$ and wish to prove it for a flat clusteron of size $N+1$. We provide a construction for general $F_{N+1,r}$, with the leftmost violinist in room $k$. 

Let us consider our operation as chip-pushing, which imposes an ordering on our violinists that is preserved across moves. In the case that $1\le r \le N-2$ and $-N+1\le k\le -3$, let us consider separately the leftmost violinist and the flat clusteron of size $N$ to its right. By our inductive assumption, we can turn the flat clusteron of size $N$ into a state with shadow $F_{N,r}$ with the leftmost violinist in room $k+2$. Performing these operations pushes the leftmost violinist into room $k+1$, upon which by Lemma~\ref{lem:nextToShadow}, the only possible attainable final state has a shadow of $F_{N+1,r}$ and the leftmost violinist in room $k$. 

In the case that $r=1$ and $k=-2$, we can perform moves on the sequence of rooms $(0,1)$, $(2,3)$, $\dots$, $(2N-8,2N-7)$, $(2N-5,2N-4)$, $(2N-7,2N-6)$, $\dots$, $(-1,0)$.

Notice that if it is possible to reach a state with shadow $F_{N+1,r}$ with the leftmost violinist in room $k$, then by performing ``inverse'' of each move (if a move on rooms $(x_i,x_i+1)$ was performed, then we instead perform a move on rooms $(N-x_i-1,N-x_i)$) we can reach a state with shadow $F_{N+1,N-r}$ with the leftmost violinist in room $-N-1-k$.

This proves that we can achieve cases with $-N+2\le k\le -2$ and $2\le r \le N-1$, as well as the case where $r=N-1$ and $k=-N+1$.

This shows that all final states with shadow $F_{N,r}$ for some $1 \le r \le N-1$ and with the leftmost violinist in room $k$ for some $-N+2\le k \le -2$ are attainable.

To show that no other final states are achievable, notice that if we do not continually perform the leftmost move, then the leftmost violinist must move to the left at least twice, once during the first move and a second time because another violinist will eventually be pushed into the room to his right. As a result, the leftmost violinist must end at the room with index $-2$ or below. We can similarly show that unless we continually perform the rightmost move, the rightmost violinist must end in a room with index $N+1$ or above. By Theorem~\ref{thm:fsclusterons}, we know that the span of any flat clusteron final state is $2N$, so the first violinist must be in some room with index between $-N+2$ and $-2$. But Theorem~\ref{thm:fsclusterons} also tells us that all flat clusteron final states have shadows in $\mathcal{F}_N$ for $N\ge 5$, which gives us the desired result.
\end{proof}

\section{Intermediate States}\label{sec:interstates}

We call an intermediate state \textit{a locked-in state} if all possible moves for the state and its future states do not change its centroid. Thus, such a locked-in state has a set final state. 

We call a state \textit{spacious}, if it does not contain any flat clusterons of size larger than $2$ and every pair of clusterons of size $2$ has at least one $2$-gap between them.

\begin{proposition}\label{prop:constantcentroid}
An intermediate state is locked-in if and only if it is spacious.
\end{proposition}

\begin{proof}
First, we show that if the state is not spacious, it is not locked-in. Suppose the state contains a $c$-clusteron with $c > 2$. Consider the move that can be performed on the two leftmost violinists inside this clusteron. The move has unequal left and right neighborhoods, and by Lemma~\ref{lem:2clusteronconstantcentroid} the centroid changes. 

Now suppose that in our state, we have two pairs of two consecutive rooms separated by $m$ 1-gaps and no 2-gaps. If we make a move on one pair of these rooms, then in the next state, we will have two pairs of two consecutive rooms separated by $m-1$ 1-gaps. If we continue making moves, then at some point, we will have three consecutive rooms, implying our state was not locked-in.

Now we show that if a state is spacious, it is locked-in. We show it is impossible to go from a spacious to a non-spacious state. This implies the result, as by Lemma~\ref{lem:2clusteronconstantcentroid}, the centroid will not change when performing moves on pairs of violinists that are part of clusterons of size $2$. 

Suppose we start with a spacious state and perform a move on a 2-clusteron $C$. Without loss of generality, we look at what happens to the right of $C$. In the worst case, there is a 2-clusteron $C_1$ to the right with the 2-gap $g_1$ in between $C$ and $C_1$. Suppose $C$ is immediately followed by a 2-gap. Then the move creates a new 2-gap in place of $C$ and no new clusterons to the right of $C$. If $C$ is not followed by a 2-gap, then the move creates a new 2-clusteron to the right of $C$, which is surrounded by a new 2-gap in place of $C$ and the gap $g_1$ to the right.
\end{proof}

Given an intermediate state, the number of possible moves depends on the number of gaps. Namely, if the number of gaps is $g$, then the number of possible moves is $N-g-1$.

On the first turn, we have $n-1$ possible moves creating one 2-gap. On the second turn, we have $n-2$ possible moves. On the third turn, we have one 2-gap and one 1-gap generating $n-3$ possible moves. After that, we always have at least one 2-gap created from the previous move and at least one other gap created from the move before that. Therefore, we always have not more than $n-3$ possible moves.

\begin{proposition}
The number of gaps can change only in the following pattern:
\begin{itemize}
    \item If the chip-firing place is surrounded by two 2-gaps or one 2-gap and the border, the number of gaps increases by 1.
    \item If the chip-firing place is surrounded by one 2-gap or the border and by one 1-gap, the number of gaps does not change.
    \item If the chip-firing place is surrounded by two 1-gaps, the number of gaps decreases by 1.
\end{itemize}
The decrease in the number of gaps can only happen starting from move 3.
\end{proposition}

\begin{proof}
    The only thing that requires a proof is the last sentence. We always have one 2-gap, and we also need two 1-gaps. Each of the gaps requires a move to create. 
\end{proof}

\begin{example}
Consider the 5-clusteron as a starting point. After the first move, we can have the shadow 1001111, which has one 2-gap. After the second move, we can have 10110011. Performing the rightmost move then gives us 101101001, upon which performing the only possible move decreases the number of gaps from 3 to 2.
\end{example}

We now show what happens when the shadows of two final states are placed next to each other. The resulting final shadow is uniquely determined, and the index of its 2-gap is the sum of the indices of the 2-gaps in the two shadows we start with.

\begin{proposition}\label{prop:mergeclusterons}
When a shadow $F_{N_1,x} \in \mathcal{F}_{N_1}$ is placed to the left and directly adjacent to a shadow $F_{N_2,y}\in \mathcal{F}_{N_2}$, their overall final state will have a shadow that can be represented by $F_{N_1+N_2,x+y}\in \mathcal{F}_{N_1+N_2}$.
\end{proposition}
\begin{proof}
Say that the leftmost room of $F_{N_1,x}$ is at position $0$, and the leftmost room of $F_{N_2,y}$ is at $2N_1$.

We can see that this satisfies the conditions described in Proposition~\ref{prop:constantcentroid}. As a result, its centroid must always stay the same. Calculation gives that the centroid lies at
$$G=\frac{1}{N_1+N_2}\cdot((0+2+\dots+2(x-1)+(2x+1)+(2x+3)+\dots+(2N_1-1)+$$
$$(2N_1)+(2N_1+2)+\dots+2(N_1+y-1)+$$
$$(2N_1+2y+1)+(2N_1+2y+3)+\dots+(2N_1+2N_2-1))$$
$$=\frac{x(x-1)+(N_1-x)(N_1+x)+y(2N_1+y-1)+(N_2-y)(2N_1+N_2+y)}{N_1+N_2}$$
$$=(N_1+N_2)-\frac{x+y}{N_1+N_2}.$$
This has fractional part $-\frac{x+y}{N_1+N_2}$, and thus the final state will have a shadow representable by $F_{N_1+N_2,x+y}\in \mathcal{F}_{N_1+N_2}$, as desired.
\end{proof}

\section{Final State Probability} \label{sec:fsprobability}

Because we can characterize all final states of flat clusterons, it is natural to ask ourselves what happens when we perform random moves starting from a flat clusteron: at each state, we select a move to make uniformly and independently out of all possible moves. It turns out that we have an equal probability of ending with a final state with any shadow in $\FF_N$. This is surprising. 

\subsection{Conjecture}

We calculated probabilities for a lot of final states and noticed an amusing (amazing) pattern. The pattern was computationally checked for up to 13 violinists. This pattern is our next conjecture.

\begin{conjecture}\label{conj:flatclusteronprobabilities}
If we start from a flat clusteron, and at each state uniformly select a move to perform from all possible moves, then the probability of ending with a final state with final shadow $F_{N,r}$ equals $\frac{1}{N-1}$ for all $1\le r\le N-1$.
\end{conjecture}

Equivalently, the probability of ending with a final state that has a centroid with a fractional value $\frac{k}{N}$ (or sumtroid with fractional value $k$) is $\frac{1}{N-1}$, where $0 \leq k < N$, with one exception at either $k=0$ if $N$ is odd or $k=\frac{N}{2}$ if $N$ is even.

Although we have not found a proof for this conjecture, we present some partial results.

We define $P_{N,K}$ to be the probability of ending with a final state with sumtroid equal to $K$, when starting from a flat clusteron of size $N$, where we assume that each flat clusteron starts with sumtroid $0$. When $N$ is even, this means that the violinists will be in rooms with indices in $\ZZ+\frac{1}{2}$, but this has negligible effects on our problem.

For reference, we provide Table~\ref{tab:probRef1} with the values of $P_{N,K}$ for small $N$. Because of symmetry $P(N,K)=P(N,-K)$, we show only the non-positive half of the table.

\begin{table}[ht!]
    \centering
    \begin{tabular}{|c|c|c|c|c|c|c|c|c|c|c|c|}
        \hline
         $N\backslash K$ & $-10$ & $-9$ & $-8$ & $-7$ & $-6$ & $-5$ & $-4$ & $-3$ & $-2$ & $-1$ & 0 \\
        \hline
         3 & & & & & & & & & & $\frac{1}{2}$ & 0\\
         \hline
         4 & & & & & & & & $\frac{1}{6}$ & 0 & $\frac{1}{6}$ & $\frac{2}{6}$\\
         \hline
         5 & & & & & $\frac{1}{24}$ & 0 & $\frac{1}{24}$ & $\frac{2}{24}$ & $\frac{4}{24}$ & $\frac{4}{24}$ & 0\\
         \hline
         6 & $\frac{1}{120}$ & 0 & $\frac{1}{120}$ & $\frac{2}{120}$ & $\frac{4}{120}$ & $\frac{8}{120}$ & $\frac{11}{120}$ & 0 & $\frac{11}{120}$ & $\frac{14}{120}$ & $\frac{16}{120}$\\
         \hline
    \end{tabular}
    \caption{Probability of ending with a sumtroid of $K$ when starting from a flat clusteron of size $N$.}
    \label{tab:probRef1}
\end{table}

\begin{example}
Consider a flat clusteron of size $4$, as in Example~\ref{ex:fClusteron4}. Each leaf in the tree in Figure~\ref{fig:4exampletree} is achieved with the same probability. As a result, in Table~\ref{tab:4fs}, we can see that the probability both for ending with a sumtroid of $-1$ and for ending with a sumtroid of $3$ is $\frac{1}{6}$. Both of them correspond to the same final shadow, which is therefore achieved with probability $\frac{1}{6}+\frac{1}{6}=\frac{1}{3}$. A similar calculation can be performed for the other shadows. Note that although in this specific case the probability of achieving the sumtroid of each final state is proportional to the number of branches corresponding to that final state, this is not always the case. 
\end{example}

We construct Table~\ref{tab:probRef2}, where its entry in row $N$ and column $K$ equals the corresponding entry in Table~\ref{tab:probRef1} multiplied by $(N-1)!$.

\begin{table}[ht!]
    \centering
    \begin{tabular}{|c|c|c|c|c|c|c|c|c|c|c|c|}
        \hline
         $N\backslash K$ & $-10$ & $-9$ & $-8$ & $-7$ & $-6$ & $-5$ & $-4$ & $-3$ & $-2$ & $-1$ & 0 \\
        \hline
         3 & & & & & & & & & & $1$ & 0\\
         \hline
         4 & & & & & & & & $1$ & 0 & $1$ & $2$\\
         \hline
         5 & & & & & $1$ & 0 & $1$ & 2 & $4$ & $4$ & 0\\
         \hline
         6 & $1$ & 0 & $1$ & $2$ & $4$ & $8$ & $11$ & 0 & $11$ & $14$ & $16$\\
         \hline
    \end{tabular}
    \caption{Probability of ending with a sumtroid of $K$ when starting from a flat clusteron of size $N$ multiplied by $(N-1)!$.}
    \label{tab:probRef2}
\end{table}

In the next lemma, we describe zeros in Table~\ref{tab:probRef2}. In particular, we show that there are $N-2$ of them in a row corresponding to $N$ violinists. Let $M=\frac{N}{2}$ for even $N$ and $0$ for odd $N$.

\begin{lemma}\label{lem:zeroState}
The sumtroid $K$ of a final state achieves all integer values in the range $-\binom{N-1}{2} \linebreak[1] \le K\le \binom{N-1}{2}$, except for $K\equiv M \pmod{N}$, in which case $P(N,K)=0$.
\end{lemma}

\begin{proof}
We know that the smallest sumtroid of a final state is achieved when we use the greedy algorithm of always performing the rightmost possible move. Thus, the smallest sumtroid is $-\binom{N-1}{2}$. By symmetry, the largest sumtroid is $\binom{N-1}{2}$.

We next show that it is impossible for there to exist a final state with sumtroid $K\equiv M \pmod{N}$, implying that $P(N,K)=0$ in these cases. Say that we end in a state with shadow $F_{N,k}$ for some $1\le k \le N-1$ with the leftmost violinist in room $r$. By Proposition~\ref{prop:centroidEnumeration}, we have
$$K=N\left(r+N-\frac{k}{N}\right)=N^2+Nr-k.$$
When $N$ is odd, then $r\in \ZZ$, and so $K\equiv -k\pmod{N}$. But by Theorem~\ref{thm:fsclusterons}, we know that $k$ can only take on values between $1$ and $N-1$ inclusive, so $K$ cannot be equivalent to $0$ modulo $N$.

When $N$ is even, then $r\in \ZZ+\frac{1}{2}$, so the term $Nr$ leaves a remainder of $\frac{N}{2}$ when divided by $N$. This tells us that $K\equiv \frac{N}{2}-k$, from which we similarly get that $K$ cannot be equivalent to $\frac{N}{2}$ modulo $N$.

We now show that $P(N,K)\neq 0$ when $-\frac{(N-1)(N-2)}{2}\le K\le \frac{(N-1)(N-2)}{2}$ and $K\not\equiv M \pmod{N}$. To prove this, we will show that there are $(N^2-3N+3)-(N-2)=N^2-4N+5$ values of $K$ that are achievable in the given range.

We once again know that $K=N^2+Nr-k$. Theorem~\ref{thm:clusteronfstrans} tells us that there exists exactly $(N-3)(N-1)+2=N^2-4N+5$ possible ending states, all of which have distinct sumtroids by Proposition~\ref{prop:centroidEnumeration}. As all of these sumtroids $K$ lie within the range $-\frac{(N-1)(N-2)}{2}\le K\le \frac{(N-1)(N-2)}{2}$, the result is implied.
\end{proof}

\subsection{Relation of probabilities to recursive trees}
\label{sec:recursivetrees}

Define a \textit{recursive tree} on $N$ vertices, labeled $0$ through $N-1$, to be a tree rooted at $0$ such that all paths starting from $0$ and ending at a leaf are strictly increasing. Such trees are also called \textit{increasing Cayley trees} \cite{Stanley}. Define the \textit{smallest rooted path} of a recursive tree to be the path starting at $0$ that always goes to the smallest child. 
We claim there exists a relationship between the number of recursive trees of size $N$ with a certain number of leaves and the smallest rooted path ending in a certain number and the probability of ending with a certain corresponding sumtroid value of a final state that has propagated from a flat clusteron of size $N$. It turns out that proving this implies Conjecture~\ref{conj:flatclusteronprobabilities}.

Let $R(N,\ell,x)$ be the number of recursive trees with $N$ vertices, $\ell$ vertices of degree $1$ (possibly including the root itself), and with the smallest rooted path ending in $x$.

\begin{conjecture}\label{conj:probBijection}
It is true that $R(N,\ell,x)=P(N,K)\cdot (N-1)!$ when 

$$\ell=\left\lfloor\frac{1}{N}\left(K+\frac{(N-2)(N-1)}{2}\right)\right\rfloor+2$$
\begin{align*}
    x = 
    \left\{
        \begin{array}{lr}
            \left(K+\frac{(N-2)(N-1)}{2}\right) \pmod{N}, & \text{if } K+\frac{(N-2)(N-1)}{2} \not\equiv 0 \pmod{N}\\
            1, & \text{if } K+\frac{(N-2)(N-1)}{2} \equiv 0 \pmod{N}
        \end{array}
    \right\}
\end{align*}
and $K\neq M \pmod{N}$.
\end{conjecture}

Here $K+\frac{(N-2)(N-1)}{2}$ is a minimal shift of our sumtroids so that every centroid is non-negative, as this makes some math/groupings easier. With this shift, the smallest sumtroid is 0. 

We now interpret the values of $\ell$ and $x$ in terms of final states. The leftmost violinist moves $N-\ell$ rooms to the left from its original state to its final state, unless the sumtroid of the final state $K\equiv -\frac{(N-2)(N-1)}{2} \pmod{N}$, in which case the leftmost violinist moves $N-\ell+1$ rooms to the left. Say that our final state is represented by the shadow $F_{N,r}$. Then $x+r \equiv 2 \pmod{N-1}$. In particular, there is a bijection between $x$ and $r$.

We can also express $K$ as
$$K=-\frac{(N-2)(N-1)}{2}+(\ell-2)N+(x-1)+\min(x-1,1).$$
Equivalently, if $x > 1$, we have
\[K=-\frac{(N-2)(N-1)}{2}+(\ell-2)N+x,\]
and for $x = 1$, we have
\[K=-\frac{(N-2)(N-1)}{2}+(\ell-2)N.\]

Table~\ref{tab:bijectionN5(2)} shows recursive trees corresponding to different sumtroids for $N=5$. The first row of the table shows the sumtroids. The second row shows the values of $(\ell,x)$ for each sumtroid. The third row shows the recursive trees corresponding to each sumtroid.
For example, when $K=-1$, we can compute that the corresponding values for $\ell$ and $x$ are $\ell=3$ and $x=1$. There are $4$ recursive trees with $3$ vertices of degree $1$ and whose smallest rooted path ends in $1$. Correspondingly, $24P(5,-1) = 4$.

Note that a sumtroid uniquely defines the number of leaves and $x$. Keep in mind that the root itself can count as a leaf too. So in the long chain, which is the leftmost tree corresponding to sumtroid $-2$, both 0 and 4 count as a leaf.

\begin{table}[ht!]
\centering
\begin{tabular}{|c|c|c|c|c|}
\hline $-6$ & $-4$ & $-3$ & $-2$ & $-1$\\
\hline (2,1) & (2,2) & (2,3) & (2,4) & (3,1)\\
\hline
\centering
$\scalebox{0.75}{\begin{forest}
    [0
    [1][2[3[4]]]]]
\end{forest}}$ &
$\scalebox{0.75}{\begin{forest}
    [0[1[2]][3[4]]]
\end{forest}}$ &
$\scalebox{0.75}{\begin{forest}
    [0[1[2[3]]][4]]
\end{forest}
\begin{forest}
    [0[1[3]][2[4]]]
\end{forest}
}$ &
$\scalebox{0.75}{\begin{forest}
    [0[1[2[3[4]]]]]
\end{forest}
\begin{forest}
    [0[1[4]][2[3]]]
\end{forest}
\begin{forest}
    [0[1[2[4]]][3]]
\end{forest}
\begin{forest}
    [0[1[3[4]]][2]]
\end{forest}}$ &
$\scalebox{0.75}{\begin{forest}
    [0[1][2][3[4]]]
\end{forest}
\begin{forest}
    [0[1][2[4]][3]]
\end{forest}
\begin{forest}
    [0[1][2[3]][4]]
\end{forest}
\begin{forest}
    [0[1][2[3][4]]]
\end{forest}}$ \\ \hline
\end{tabular}

\label{tab:bijectionN5}
\end{table}

\begin{table}[ht!]
\centering
\begin{tabular}{|c|c|c|c|c|}
\hline $1$ & $2$ & $3$ & $4$ & $6$\\
\hline (3,2) & (3,3) & (3,4) & (4,1) & (4,2)\\
\hline
\centering
$\scalebox{0.6}{\begin{forest}
    [0[1[2]][3][4]]
\end{forest}
\begin{forest}
    [0[1[2][3]][4]]
\end{forest}
\begin{forest}
    [0[1[2][4]][3]]
\end{forest}
\begin{forest}
    [0[1[2][3[4]]]]
\end{forest}}$ &
$\scalebox{0.6}{\begin{forest}
    [0[1[3]][2][4]]
\end{forest}
\begin{forest}
    [0[1[3][4]][2]]
\end{forest}
\begin{forest}
    [0[1[2[3]][4]]]
\end{forest}
\begin{forest}
    [0[1[2[3][4]]]]
\end{forest}}$ &
$\scalebox{0.6}{\begin{forest}
    [0[1[4]][2][3]]
\end{forest}
\begin{forest}
    [0[1[2[4]][3]]]
\end{forest}}$ &
$\scalebox{0.6}{\begin{forest}
    [0[1][2][3][4]]
\end{forest}}$ &
$\scalebox{0.6}{\begin{forest}
    [0[1[2][3][4]]]
\end{forest}}$ \\ \hline
\end{tabular}
\caption{A table showing the relationship in Conjecture~\ref{conj:probBijection}. The number of trees in each column matches the corresponding entry in row $N=5$ of Table~\ref{tab:probRef2}.}
\label{tab:bijectionN5(2)}
\end{table}

We later have a theorem that describes the recursion that the values $R(N,\ell,i)$ follow. But before this, we prove a lemma.

Let $A(N,\ell,x)$ be the number of trees having $N$ vertices, $\ell$ leaves such that the root is not a leaf, and the smallest rooted path ending in $x$. Let $B(N,\ell,x)$ be the number of trees having $N$ vertices, $\ell$ leaves such that the root is a leaf, and the smallest rooted path ending in $x$. 

\begin{lemma}
    The numbers $A(N,\ell,x)$ and $B(N,\ell,x)$ satisfy the following equations:
\begin{equation}
\label{eq:complementarysubsets}
A(N,\ell,x) + B(N,\ell,x) = R(N,\ell,x).
\end{equation}
\begin{equation}
\label{eq:rootLeafBijection}
B(N,\ell,x) = A(N-1,\ell-1,x-1)+B(N-1,\ell,x-1).
\end{equation}
\begin{equation}
\label{eq:BSumtoA1}
A(N,\ell,1)=\sum_{i=2}^{N-1} B(N,\ell, i).
\end{equation}
In addition, for $N > 2$, we have $B(N,\ell, 1) = 0$.
\end{lemma}

\begin{proof}
Eq.~(\ref{eq:complementarysubsets}) describes the fact that $A(N,\ell,x)$ and $B(N,\ell,x)$ count complementary subsets of the set of all recursive trees with given parameters.

Now we prove Eq.~(\ref{eq:rootLeafBijection}). Consider a tree with its root being a leaf. Then we can collapse the root and its child into one vertex. Depending on the number of neighbors that the original root's child had, the new root will either no longer be a leaf or still be a leaf. In the first case, both the number of vertices and the number of leaves decrease by $1$. In the second case, the number of vertices decreases by $1$, while the number of leaves remains unchanged. In both cases, all labels also need to decrease by $1$ (to make the tree $0$-rooted).

We prove Eq.~(\ref{eq:BSumtoA1}) by defining the corresponding bijection. Indeed, if we let $C_0$ be the set of all children of the root of a tree, then for every tree with its root not being a leaf and having the smallest rooted path ending in $1$, we can switch the parent of all leaves in $C_0 \setminus \{1\}$ to the leaf with label $1$ (while preserving all other connections) to create a unique tree with its root a leaf and $\ell$ total leaves (as $0$ is now a leaf, but $1$ is no longer a leaf).

And if we let $C_1$ be the set of all children with parent $1$, for every tree with its root a leaf and $\ell$ leaves, we can switch the parent of all leaves in $C_1$ to the root of the tree to create a unique tree with its root a leaf, $\ell$ total leaves (as $1$ is now a leaf, but $0$ is no longer a leaf), and the smallest rooted path ending in $1$.

The last equation follows from the fact that if the vertex labeled 1 is a leaf, the root must have other children and cannot be a leaf, provided that the total number of vertices is more than 2.
\end{proof}

We can get that $R(2,2,1)=1$. We are now ready for our theorem.
\begin{theorem} \label{thm:recursiontrees}
For $N > 2$ we have
\begin{equation}
\label{eq:Rrecurrence}
R(N,\ell,x)=\sum_{i=\max(x,2)}^{N-2} R(N-1,\ell-1,i) + \sum_{i=1}^{\max(x-1,1)} R(N-1,\ell,i).
\end{equation}
\end{theorem}

\begin{proof}
The proof is by induction. We can construct a recursive tree with $N$ vertices, $\ell$ vertices of degree $1$, and the smallest rooted path ending in $x$ by adding a vertex to a recursive tree with $N-1$ vertices.

We consider a recursive tree with $N-1$ vertices and relabel its vertices with the numbers $0,1,\dots,x-1,x+1,\dots, N-1$ while keeping relative order so that we exclude vertex $x$.

Suppose $x > 1$, and our tree with $N-1$ vertices has the smallest rooted path ending in $i < x$. Then we can attach a vertex with label $x$ to vertex $i$. In this case, the number of leaves does not change.

If $i > x$, we can attach a vertex with label $x$ to the earliest ancestor of vertex $i$, which is less than $x$. If this earliest ancestor is zero, then, given that $x > 1$, zero cannot be a leaf. Thus, we add one leaf without destroying any existing leaves.

For every $i$, there are 
$R(N,\ell-1,i)$
recursive trees of the first type and $R(N,\ell,i)$ of the second type. This gives for $x > 1$ the following formula, matching what we want to prove,
$$R(N,\ell,x)=\sum_{i=x}^{N-2} R(N-1,\ell-1,i) + \sum_{i=1}^{x-1} R(N-1,\ell,i).$$

If $x$ is 1, we have to attach $x$ to the root. Suppose $T$ is the tree to which we attach $x$ after relabeling its vertices. If the root in $T$ is not a leaf, we add a new leaf. If the root is a leaf, then we do not change the number of leaves. We get
\begin{equation}
\label{eq:RtoA,B}
R(N,\ell,1) = \sum_{i=1}^{N-2} A(N-1,\ell-1,i)+\sum_{i=1}^{N-2} B(N-1,\ell,i).
\end{equation}
Noting that for $N > 2$, $B(N,\ell, 1) = 0$, we have
\[R(N,\ell,1) = \sum_{i=1}^{N-2} A(N-1,\ell-1,i)+\sum_{i=2}^{N-2} B(N-1,\ell,i).\]
Using Eq.~(\ref{eq:rootLeafBijection}) we get
\begin{multline*}
   R(N,\ell,1) = \sum_{i=1}^{N-2} A(N-1,\ell-1,i)+\sum_{i=2}^{N-2} \left(A(N-2,\ell-1,i-1)+B(N-2,\ell,i-1)\right)\\
   =\sum_{i=1}^{N-2} A(N-1,\ell-1,i)+\sum_{i=1}^{N-3} A(N-2,\ell-1,i)+ \sum_{i=1}^{N-3}B(N-2,\ell,i) \\
   = A(N-1,\ell-1,1)+\sum_{i=2}^{N-2} A(N-1,\ell-1,i)+\sum_{i=1}^{N-3} A(N-2,\ell-1,i)+A(N-2,\ell,1). 
\end{multline*}
Applying Eq.~(\ref{eq:BSumtoA1}) to the first and the last term and then use Eq.~(\ref{eq:complementarysubsets})
to combine the first two sums
\begin{multline*}
R(N,\ell,1) = \sum_{i=2}^{N-2}B(N-1,\ell-1,i)+\sum_{i=2}^{N-2} A(N-1,\ell-1,i)+\\
\sum_{i=1}^{N-3} \left(A(N-2,\ell-1,i)+B(N-2,\ell,i)\right) = \\
\sum_{i=2}^{N-2}R(N-1,\ell-1,i)+
\sum_{i=1}^{N-3} \left(A(N-2,\ell-1,i)+B(N-2,\ell,i)\right).
\end{multline*}
and after applying Eq.~(\ref{eq:RtoA,B}) to the last summation, we get
\[R(N,\ell,1) = R(N-1,\ell,1)+\sum_{i=2}^{N-2} R(N-1,\ell, i),\]
as desired.

Combining the results for $x>1$ and $x=1$, we get
$$R(N,\ell,x)=\sum_{i=\max(x,2)}^{N-2} R(N-1,\ell-1,i) + \sum_{i=1}^{\max(x-1,1)} R(N-1,\ell,i).$$
\end{proof}

This gives us many details about the behavior of $P(N,K)$. A corresponding recursion seems to hold for the values in Table~\ref{tab:probRef2}, as shown below, in which each element in the $(N+1)$-st row is generated by summing $N$ consecutive elements from the row above it, in a fashion similar to that of a sliding window, after which zeroes are added in accordance to Lemma~\ref{lem:zeroState}.

\begin{conjecture} \label{conj:fsprobabilityrecur}
The recurrence Eq.~(\ref{eq:Rrecurrence}) can be described by the following recurrence relation for $K\not\equiv M \pmod{N}$:
$$P(N,K)=\sum_{i=K-\frac{N-1}{2}-A}^{K+\frac{N-1}{2}-A-1} P(N-1,i)$$
where $A=\left\lfloor\frac{1}{N}(K+M)\right\rfloor-\frac{1+(-1)^{N}}{4}$.
\end{conjecture}

See Figure~\ref{fig:slidingWindowBiject} for an example of how one can use the values in row $N=4$ of the table to generate values in row $N=5$ and for an example of how one can use the values in row $N=5$ of the table to generate values in row $N=6$. An example of how an entry is generated is marked in a different color and underlined in both cases.

Summing from $K-\frac{N-1}{2}$ to $K+\frac{N+1}{2}-1$ in Conjecture~\ref{conj:fsprobabilityrecur} represents us adding $N-1$ values from the $(N-1)$-st row to obtain a value in the $N$-th row of the table, while the value of $A$ that is subtracted from both the top and bottom of the summation fixes the indexing, as it corresponds to inserting zeroes.

\begin{figure}[ht!]
\centering
\begin{tikzpicture}[node distance=2cm]

\tikzstyle{startstop} = [rectangle, rounded corners, minimum width=3cm, minimum height=1cm,text centered, draw=black]
\tikzstyle{io} = [trapezium, trapezium left angle=70, trapezium right angle=110, minimum width=3cm, minimum height=1cm, text centered, draw=black, fill=blue!30]
\tikzstyle{arrow} = [thick,->,>=stealth]
\tikzstyle{line} = [draw, -latex']

\node (start) [startstop] {1 \color{red} \underline{\textbf{0 1 2 1}} \color{black} 0 1};
\node (in1) [below of=start] {1 1 2 4 \color{red} \underline{\textbf{4}} \color{black} 4 4 2 1 1};
\node (stop) [startstop, below of=in1] {1 0 1 2 4 \color{red} \underline{\textbf{4}} \color{black} 0 4 4 2 1 0 1};

\node (start2) [startstop, right of=start, xshift=6cm] {1 0 \color{blue} \underline{\textbf{1 2 4 4 0}} \color{black} 4 4 2 1 0 1};
\node (in2) [right of=in1, xshift=6cm] {1 1 2 4 8 11 \color{blue} \underline{\textbf{11}} \color{black} 14 16 14 11 11 8 4 2 1 1};
\node (stop2) [startstop, right of=stop, xshift=6cm] {1 0 1 2 4 8 11 0 \color{blue} \underline{\textbf{11}} \color{black} 14 16 14 11 0 11 8 4 2 1 0 1};

\path [line] (start) -- node [text width=6cm,midway ] {\color{red}\textbf{0+1+2+1=4}\color{black}} (in1);
\path [line] (in1) -- node [text width=5cm,midway ] {Insert zeroes} (stop);

\path [line] (start2) -- node [text width=7.5cm,midway ] {\color{blue}\textbf{1+2+4+4+0=11}\color{black}} (in2);
\path [line] (in2) -- node [text width=5cm,midway ] {Insert zeroes} (stop2);
\end{tikzpicture}
\caption{The recursion used to generate the rows of Table~\ref{tab:probRef2}.}
\label{fig:slidingWindowBiject}
\end{figure}
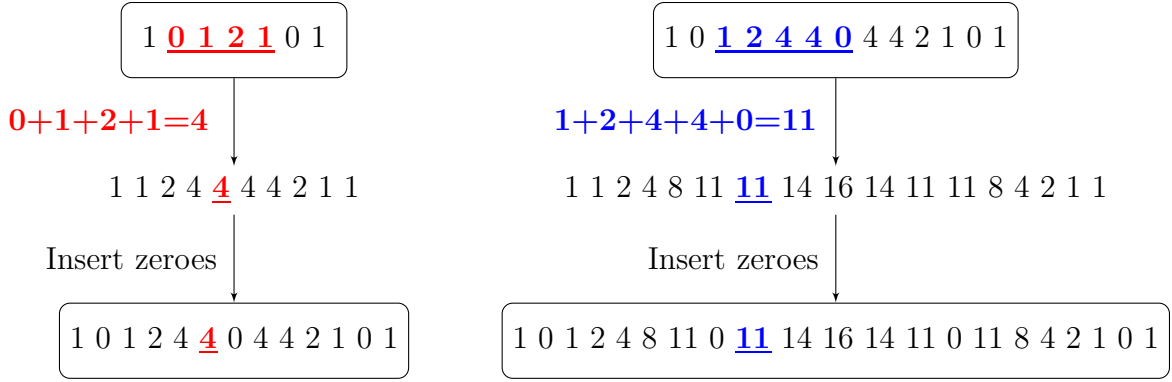

And if we define
\begin{equation}
\label{eq:TNell}
  T(N,\ell)=\sum_{i=0}^{N-1} R(N,\ell,i),  
\end{equation}
then $T(N,\ell)$ is the number of recursive trees with $N$ vertices and $\ell$ leaves. It is known that $T(N-1,\ell-2)$ equals twice the sequence of $2$-Eulerian numbers, which is sequence A120434 in the OEIS \cite{OEIS}. The 2-Eulerian numbers themselves are sequence A144696. Importantly, there exists an explicit formula for $T(N,\ell)$, see \cite{Conger},
\[T(N,\ell) = \sum_{j=0}^{\ell-2} (-1)^j\cdot (\ell -1 - j)\cdot \binom{N}{j}\cdot (\ell - j)^N.\]

In terms of violinists, $T(N,\ell)$ represents $(N-1)!$ times the sum of the probabilities of achieving the final states whose corresponding recursive trees have $\ell$ leaves.

We can also express the sum $T(N,\ell)$ in terms of centroids:
$$T(N,\ell)=\sum_{i=-(N-1)}^{N-1} (N-1-|i|)\cdot P\left(N-1,-\frac{(N-2)(N-3)}{2}+(N-1)(N-1-\ell)+i\right).$$

\begin{example}
Suppose that $N=5$, and we wish to compute $T(N,\ell)$ over all possible values of $\ell$. Then the $\ell=4$ case corresponds to sumtroids $-6$, $-4$, $-3$, and $-2$, the $\ell=3$ case corresponds to sumtroids $-1$, 1, 2, and 3, and the $\ell=2$ case corresponds to sumtroids 4 and 6. The summation gives values of 8, 14, and 2, respectively.
\end{example}

To show how our three conjectures are connected, we need the following lemma.

\begin{lemma}\label{lem:smallestPathEnd}
    The number of recursive trees with $N>1$ vertices and the smallest rooted path ending in $x$ is $(N-2)!$ for any $x$ between $1$ and $N-1$, inclusive.
\end{lemma}
\begin{proof}
For $N=2$, there is one recursive tree: a path connecting two vertices labeled 0 and 1. The smallest rooted path of this tree ends in 1. Thus, the statement holds.

We proceed with induction. We claim that if there is an equal probability, namely $\frac{1}{N-1}$, of some randomly and uniformly selected recursive tree of size $N$ to have its smallest rooted path end in $x$ for each of $1\le x\le N-1$, then we can say the same for a recursive tree of size $N+1$. This finishes the proof, as the number of recursive trees of size $N$ and with smallest rooted path ending in $x$ would be $\frac{1}{N-1}(N-1)!=(N-2)!$.

Choosing a random recursive tree of size $N+1$ is the same as choosing a random recursive tree of size $N$ with the smallest rooted path ending in $x$ and attaching the vertex with label $N$ randomly and uniformly to one of the other $N$ vertices. Because $N$ has a label larger than all other vertices, the probability that $N$ does not attach to vertex $x$ and does not change the value at the end of the smallest rooted path is $\frac{N-1}{N}$, whereas the probability that it does attach to vertex $x$ and become the new value at the end of the smallest rooted path is $\frac{1}{N}$. And by our inductive assumption, the probability that the value at the end of the smallest rooted path is less than $N$ is also equal to $\frac{1}{N-1}\cdot \frac{N-1}{N}=\frac{1}{N}$ for each distinct positive integer less than $N$.
\end{proof}

\begin{theorem}
Conjectures~\ref{conj:flatclusteronprobabilities} and~\ref{conj:fsprobabilityrecur} follow from Conjecture~\ref{conj:probBijection}.
\end{theorem}

\begin{proof}
    Assuming that Conjecture~\ref{conj:probBijection} is true, we have that
    $$\sum_{i \equiv r \pmod{N}} P(N,i) = \frac{1}{(N-1)!}\sum_{j=2}^{N-1} R(N,j,x),$$
    where $x$ is defined as in Conjecture~\ref{conj:probBijection}.
    But by Lemma~\ref{lem:smallestPathEnd}, the summation on the right-hand side of this equation equals $(N-2)!$, which gives that the expression on the right-hand side equals $\frac{1}{(N-1)!}(N-2)!=\frac{1}{N-1}$.
    Conjecture~\ref{conj:fsprobabilityrecur} follows immediately from Theorem~\ref{thm:recursiontrees} and applying the bijection from Conjecture~\ref{conj:probBijection}.
\end{proof}

The numbers $R(N,\ell,x)$ appear in the setting of permutations on $[N-1]$, where $N-1$ represents the size of the permutation. Namely, the number of permutations of $N-1$ elements with $\ell-1$ descents ending in $x-1$ equals $R(N,\ell,x)$. These numbers are the same because they follow the same recursion and have the same initial values \cite{Conger}. In the next section, we connect them directly.

\section{Permutations}
\label{sec:permutations}

\subsection{Preliminary definitions}

We start with some definitions that are used throughout this section. 

Given a permutation $\sigma$ of $n$ elements, an \textit{ascent} is any position $i < n$, where the following value is bigger than the current one. That is, if $\sigma = \sigma_1\sigma_2 \ldots \sigma_n$, then $i$ is an ascent if $\sigma_i < \sigma_{i+1}$

Similarly, a \textit{descent} is a position $i < n$ with $\sigma_i > \sigma_{i+1}$. A \textit{big descent} is a position $i < n$ with $\sigma_i - \sigma_{i+1} \geq 2$.

\subsection{Bijection between recursive trees and permutations}

In this subsection, we describe the bijection between recursive trees and permutations.

Given a recursive tree, start with the root, which has label zero. Find the root's largest child and write down the number it is labeled with. Consider the subtree starting with the largest child and repeat recursively. After finishing iterating over all points in the subtree, go to the next largest child. We know that this map is a bijection by \cite{Stanley}. Moreover, the end of the smallest rooted path corresponds to the last digit in the corresponding permutation.

We need an extra definition for the proof of Theorem~\ref{thm:permutations2trees}. We call index $i$, a \textit{special descent} of a given permutation, if $i > 0$ and is a descent in the permutation, or if $i = 0$, and the permutation starts with 1.

We now show the numbers $R(N,\ell,x)$ relate to permutation due to \cite{Conger}. In the next theorem, we produce an explicit bijection relating recursive trees to permutation.

\begin{theorem}
\label{thm:permutations2trees}
The number of permutations of order $N-1$ with $\ell-1$ descents and last digit $x-1$ equals $R(N,\ell,x)$ --- the number of recursive trees with $N$ vertices, $\ell$ leaves, and the smallest rooted path ending in $x$.
 \end{theorem}

 We prove this in several steps. We start with the following lemma.

 \begin{lemma}
\label{lemma:permutations2trees}
The number of permutations of order $N-1$ with $\ell-1$ special descents and last digit $x$ equals $R(N,\ell,x)$ --- the number of recursive trees with $N$ vertices, $\ell$ leaves, and the smallest rooted path ending in $x$.
 \end{lemma}

\begin{proof}
Following our bijection between trees and permutations, we see that each descent corresponds to an extra leaf, as we create an extra branch. The number of leaves is the number of descents plus 1. In addition, we have one extra leaf if the permutation starts with 1, as the root forms a leaf in this case. And as we know, the end of the smallest rooted path of a tree corresponds to the last digit in the corresponding permutation. Thus, $R(N,\ell,x)$ equals the number of permutations of order $N-1$ with $\ell-1$ special descents and last digit $x$.
\end{proof}

Table~\ref{tab:correspondN5} shows the permutations that correspond to the given sumtroids for $N=5$, using the algorithm described above. Each ordered pair $(\ell,x)$ in the second row means the permutation has $\ell$ special descents and ends with number $x$.

\begin{table}[ht!]
\centering
\begin{tabular}{|c|c|c|c|c|c|c|c|c|c|}
\hline $-6$ & $-4$ & $-3$ & $-2$ & $-1$ & $1$ & $2$ & $3$ & $4$ & $6$\\
\hline (1,1) & (1,2) & (1,3) & (1,4) & (2,1) & (2,2) & (2,3) & (2,4) & (3,1) & (3,2)\\
\hline
\centering
2341 & 3412 & 4123 & 1234 & 3421 & 4312 & 4213 & 3214 & 4321 & 1432\\
     &      & 2413 & 2314 & 3241  & 4132 & 2143 & 1324 &      & \\
     &      &      & 3124 & 4231  & 3142 & 1423 &      &      & \\
     &      &      & 2134 & 2431  & 1342 & 1243 &      &      &    \\ \hline
\end{tabular}
\caption{A corresponding table to Table~\ref{tab:bijectionN5(2)}, but with permutations and descents.}
\label{tab:correspondN5}
\end{table}

Interestingly, 2-Eulerian numbers appear in permutations in many ways, see a comment to A120434 in OEIS \cite{OEIS}. We discuss one of the ways in this subsection.

It is known that $T(N,k)$ gives the number of permutations of size $N+1$ starting with 2 and having $k+1$ descents \cite{Conger}. We can expend this to our numbers.

\begin{lemma}
\label{lemma:permutationsstartingwith2}
The numbers $R(N,\ell,i)$ can be described as the number of permutations of order $N$ starting with 2, with $\ell-1$ descents ending in $i+1$, if $i > 1$. Numbers $R(N,\ell,1)$ can be described as the number of permutations of order $N$ starting with 2, with $\ell-1$ descents ending in $1$.
\end{lemma}

\begin{proof}
By Lemma~\ref{lemma:permutations2trees}, the number of permutations of order $N$ with $\ell$ special descents ending in $i$ is 
the same as the number of recursive trees with $N+1$ vertices, $\ell+1$ leaves, and the smallest rooted path ending in $x$, which by definition is $R(N+1,\ell+1,x)$.

Consider a permutation of order $N+1$ starting with 2 and having $\ell$ descents. Suppose we remove the starting number and lower all other numbers in the permutation that are greater than 2 by 1. We get a permutation of order $N$.
    
Now we see what happens with descents. Ignoring the first two numbers in our permutations, ascents and descents stay in place. In addition, if the second number of the original permutation was 1, we lose a descent. Thus, the number of descents in the original permutation equals the number of special descents in the new one.

This tells us that the number of permutations of order $N+1$ starting with 2, having $\ell$ descents and ending with $x > 2$ is the same as the number of permutations of order $N$ with $\ell$ special descents and ending in $x-1$. Thus, this number is $R(N+1,\ell+1,x-1)$.

Similarly, the number of permutations of order $N+1$ starting with 2, having $\ell$ descents, and ending with $1$ is the same as that of order $N$ with $\ell$ special descents and ending in $1$. Thus, this number is $R(N+1,\ell+1,1)$. The lemma follows.
\end{proof}

\begin{proof}[Proof of Theorem~\ref{thm:permutations2trees}.]
    Using Lemma~\ref{lemma:permutationsstartingwith2}, number $R(N,\ell,x)$ equals the number of permutations of order $N$ with $\ell-1$ descents with first digit $2$ and last digit $x+1$. The number of such permutations equals the number of permutations of order $N$ with $\ell-1$ descents with first digit $1$ and last digit $x$ by \cite{Conger}, after which dropping the ``1'' at the start and decreasing all numbers in the permutation by $1$ gives us that this is equal to the number of permutations of order $N-1$ with $\ell-1$ descents last digit $x-1$, as desired.
\end{proof}

Conjecture~\ref{conj:flatclusteronprobabilities} in terms of permutations is especially simple. It states that the number of permutations ending in $i$ is the same for any $i$.

\subsection{The bijection changing the sumtroid sign}

We have a natural bijection on the initial state of violinists. If we assume the centroid of the initial state to be zero, then for a given set of moves performed on violinists in which the $i$-th move is made on violinists in positions $a_i$ and $a_i+1$, we can instead perform moves on violinists in positions $-a_i$ and $-a_i-1$. Thus, we also have a bijection on the final states--- changing the sign of the sumtroid will not change the probability that we end up in a final state with that sumtroid.

Now consider the following bijection on permutations: swap 1 and 2, and for $i > 2$, swap $i$ and $n+3-i$. Notice that this is equivalent to showing:
\[R(N,\ell,1) = R(N,N+1 - \ell,2)\]
and for $i > 2$
\[R(N,\ell,i) = R(N,N- \ell,N+2 - i),\]
for which the proof can be found in \cite{Conger}.

\subsection{Eulerian numbers}

Our numbers are related to 2-Eulerian numbers as seen in Eq.(\ref{eq:TNell}). Here, we describe the connection of our numbers $R(N,\ell,i)$ to 1-Eulerian numbers, usually called Eulerian numbers.

We can write an independent recursion for $R(N,\ell,1)$. This recursion is known for permutations \cite{Conger}, but we want to give a proof in terms of recursive trees.

\begin{proposition}
The values $R(N,\ell,1)$ follow the recursion
\[R(N,\ell,1) = (N+1 - \ell)R(N-1,\ell,1) + (\ell - 2)R(N-1,\ell-1,1).\] 
\end{proposition}

\begin{proof}
We start with a tree with $N$ vertices with the smallest rooted path ending in 1. Consider what happens when we attach a new vertex labeled $N$ to the tree. For our new tree to still have the smallest rooted path ending in $1$, we can attach $N$ to any vertex, except for the vertex labeled 1. If we attach it to a leaf, we do not create more leaves. If we attach it to a non-leaf, we create an extra leaf. Thus, 
\[R(N,\ell,1) = (N+1 - \ell)R(N-1,\ell,1) + (\ell - 2)R(N-1,\ell-1,1).\] 
\end{proof}

This tells us that numbers $R(N,\ell,1)$ follow the same recursion as Eulerian numbers \cite{Stanley}. Comparing the initial values we get $R(N,\ell,1) = A(N,\ell)$ for $N>2$. See also sequence A008292 in OEIS \cite{OEIS}.

Consider the following corresponding example in terms of permutations.

\begin{example}
Consider the number of permutations of size $n$ with $\ell$ special descents and ending in 1. Such a permutation for $n >1$ cannot start with 1. Thus, the number of such permutations is the number of permutations ending in 1 with $\ell$ descents. If we remove the last number in the permutation and lower all other numbers by 1, we get a permutation of order $n-1$ with $\ell-1$ descents. It is well known \cite{Stanley} that the number of such permutations is an Eulerian number $A(n-1,\ell)$.
\end{example}

\section{Additional Data}
\label{sec:additionaldata}

In this section, we present additional data that we calculated.

\subsection{States with five violinists}

Figure~\ref{fig:5exampletree} shows the left half of the tree diagram for possible moves for 5 violinists. The vertex labels are sumtroids. 
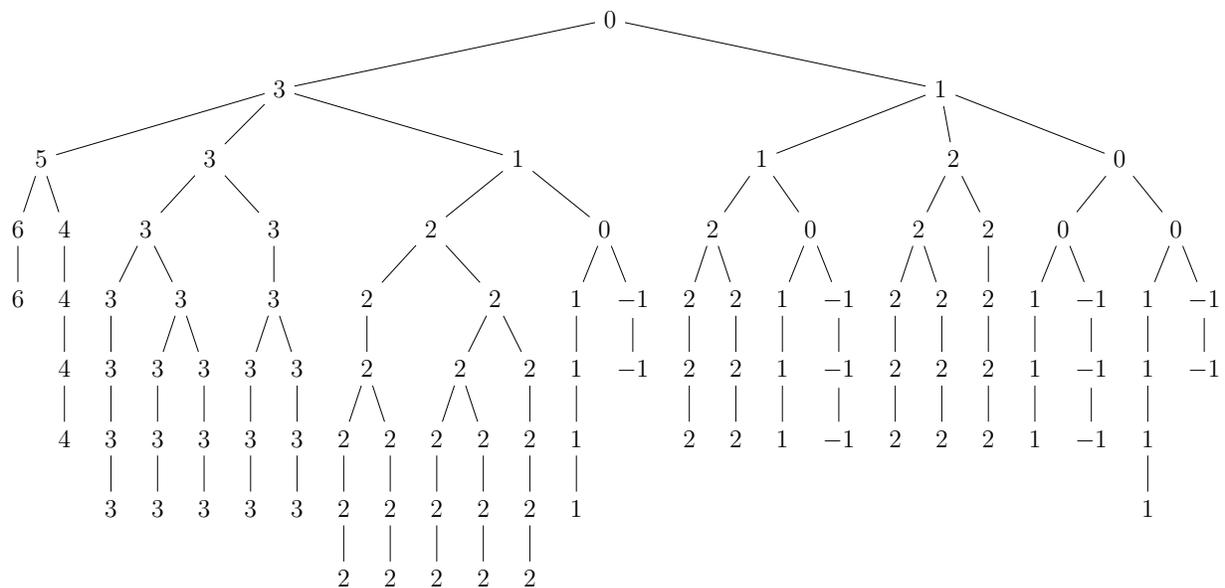
\begin{figure}[ht!]
    \centering
        \resizebox{\textwidth}{!}{
            \begin{forest}
                    [$0$[$3$[$5$[$6$[$6$]][$4$[$4$[$4$[$4$]]]]][$3$[$3$[$3$[$3$[$3$[$3$]]]][$3$[$3$[$3$[$3$]]][$3$[$3$[$3$]]]]][$3$[$3$[$3$[$3$[$3$]]][$3$[$3$[$3$]]]]]][$1$[$2$[$2$[$2$[$2$[$2$[$2$]]][$2$[$2$[$2$]]]]][$2$[$2$[$2$[$2$[$2$]]][$2$[$2$[$2$]]]][$2$[$2$[$2$[$2$]]]]]][$0$[$1$[$1$[$1$[$1$]]]][$-1$[$-1$]]]]][$1$[$1$[$2$[$2$[$2$[$2$]]][$2$[$2$[$2$]]]][$0$[$1$[$1$[$1$]]][$-1$[$-1$[$-1$]]]]][$2$[$2$[$2$[$2$[$2$]]][$2$[$2$[$2$]]]][$2$[$2$[$2$[$2$]]]]][$0$[$0$[$1$[$1$[$1$]]][$-1$[$-1$[$-1$]]]][$0$[$1$[$1$[$1$[$1$]]]][$-1$[$-1$]]]]]
                    ]
            \end{forest}
        }
        \caption{Tree diagram for possible moves for 5 violinists.}       
    \label{fig:5exampletree}
\end{figure}

Figure~\ref{fig:5exampletreenodup} shows the same diagram with the states that are children of locked-in states removed.
\label{ex:fClusteron5Simplified}
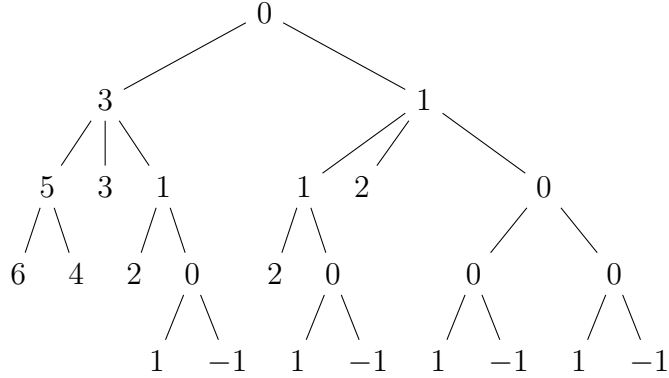
\begin{figure}[ht!]
    \centering
        \begin{forest}
             [$0$[$3$[$5$[$6$][$4$]][$3$][$1$[$2$][$0$[$1$][$-1$]]]][$1$[$1$[$2$][$0$[$1$][$-1$]]][$2$][$0$[$0$[$1$][$-1$]][$0$[$1$][$-1$]]]]
             ]
        \end{forest}
    \caption{Left half of the tree for flat clusteron of size 5 minus duplicates.}
    \label{fig:5exampletreenodup}
\end{figure}

\subsection{Probabilities for 6 to 9 violinists}
We also add the probability of ending with each centroid for $N=6$, $N=7$, $N=8$, and $N=9$ below. The probabilities are multiplied by $(N-1)!$ to turn them into integers. The numbers do not fit into a table, so we present them as a sequence.  

For $N = 6$ we have the following sequence corresponding to sumtroids from $-10$ to 0:
\[1,\ 0,\ 1,\ 2,\ 4,\ 8,\ 11,\ 0,\ 11,\ 14,\ 16.\]

For $N = 7$ we have the following sequence corresponding to sumtroids from $-15$ to 0: 
\[1,\ 0,\ 1,\ 2,\ 4,\ 8,\ 16,\ 26,\ 0,\ 26,\ 36,\ 48,\ 60,\ 66,\ 66,\ 0.\]

For $N = 8$ we have the following sequence corresponding to sumtroids from $-21$ to 0: 
\[1,\ 0,\ 1,\ 2,\ 4\, 8,\ 16,\ 32,\ 57,\ 0,\ 57,\ 82,\ 116,\ 160,\ 212,\ 262,\ 302,\ 0,\ 302,\ 342,\ 372,\ 384.\]

For $N = 9$ we have the following sequence corresponding to sumtroids from $-28$ to 0: 
\[1,\ 0,\ 1,\ 2,\ 4,\ 8,\ 16,\ 32,\ 64,\ 120,\ 0,\ 120,\ 176,\ 256,\ 368,\ 520,\ 716,\ 946,\]
\[1191,\ 0,\ 1191,\ 1436,\ 1696,\ 1952,\ 2176,\ 2336,\ 2416,\ 2416,\ 0.\]

\section{Acknowledgments}
We are grateful to Darij Grinberg for suggesting this project and for helpful consultations. We thank Ira Gessel and Richard Stanley for answering our questions, and to David Dong for helpful comments on our first draft. We thank MIT PRIMES-USA for giving us the opportunity to conduct this research.


\begin{thebibliography}{99}

\bibitem{chipFire2} Richard Anderson, L\'aszl\'o Lov\'asz, Peter Shor, Joel
Spencer, \'Eva Tardos, and Shmuel Winograd, Disks, balls,
and walls: analysis of a combinatorial game, \textit{Amer.\ Math.\ Monthly} 96.6 (1989), 481--493. MR 999411

\bibitem{abelianSandpile1} Per Bak, Chao Tang, and Kurt Wiesenfeld, Self-organized
criticality, \textit{Phys. Rev.} A (3) 38.1 (1988), 364--374.
MR 949160 (89g:58126)

\bibitem{chipFire3} Anders Bj\"orner, L\'aszl\'o Lov\'asz, and Peter W.~Shor, Chip-firing games on graphs, \textit{European J. Combin}. 12.4 (1991), 283--291. MR 1120415 (92g:90193)

\bibitem{Conger} Conger, M. A. (2010). A refinement of the Eulerian numbers, and the joint distribution of $\pi$(1) and Des($\pi$) in
Sn. Ars Combin., 95, 445–472.

\bibitem{DarijGrinberg2021} Darij Grinberg, Math 235: Mathematical Problem Solving, unpublished manuscript, available at \url{https://www.cip.ifi.lmu.de/~grinberg/t/20f/mps.pdf}, 2021.

\bibitem{abelianSandpile2} Deepak Dhar, Self-organized critical state of sandpile automaton models, \textit{Phys.\ Rev.\ Lett.} 64.14 (1990), 1613--1616. MR 1044086 (90m:82053)

\bibitem{sortChipFire} Sam Hopkins, Thomas McConville, and James Propp, Sorting via chip-firing, \textit{Electron.\ J.\ Comb.}, 24, (2016).

\bibitem{confluence1} M.~H.~A.~Newman. On theories with a combinatorial definition of “equivalence”, \textit{Ann.\ Math.}, 43 (1942) 223--243.

\bibitem{OEIS} OEIS Foundation Inc. (2023), The On-Line Encyclopedia of Integer Sequences, Published electronically at \url{https://oeis.org.}

\bibitem{chipFireStart} J.~Spencer, Balancing vectors in the max norm, \textit{Combinatorica} 6.1 (1986), 55--65. MR 856644

\bibitem{Stanley} Richard, P.~Stanley, \textit{Enumerative Combinatorics}. Cambridge studies in Advanced Mathematics 1 (2011).
\end{thebibliography}
\end{document}